\documentclass[reqno, 11pt, a4paper, oneside]{amsart}
\usepackage{amsmath}
\usepackage{amsfonts}
\usepackage{amssymb}
\usepackage[ansinew]{inputenc}
\usepackage{graphicx}
\usepackage{amsbsy}
\usepackage{fancyhdr}
\usepackage{yfonts}
\usepackage{hyperref}
\usepackage{enumerate}

\numberwithin{equation}{section}
\newtheorem{teo}{Theorem}[section]
\newtheorem{prop}[teo]{Proposition}
\newtheorem{lema}[teo]{Lemma}

\theoremstyle{definition}

\theoremstyle{remark}

\title{Bounded rational points on curves}
\author{Miguel N. Walsh}
\address{Departamento de Matemática, Facultad de Ciencias Exactas y Naturales, Universidad de Buenos Aires, 1428 Buenos Aires, Argentina}
\email{mwalsh@dm.uba.ar}
\thanks{The author was partially supported by a CONICET doctoral fellowship.}

\begin{document}

\def\Fqn{\mathbb{F}_q^n}
\def\Fq{\mathbb{F}_q}
\def\Di{\mathbb{D}}
\def\O{\mathcal{O}}

\begin{abstract}
We establish the sharp estimate $\ll_d N^{2/d}$ for the number of rational points of height at most $N$ on an irreducible projective curve of degree $d$. We deduce this from a result for general hypersurfaces that is sensitive to the coefficients of the corresponding form.
\end{abstract}

\maketitle

\section{Introduction}

We are concerned with giving uniform estimates for the number of rational points on an irreducible projective curve $C$ of degree $d$. Writing $X(C;N)$ for the number of rational points of height $ \le N$, from the consideration of rational curves we know that this quantity can grow like $N^{2/d}$. In this article we establish the conjectured tightness of this bound.

\begin{teo}
\label{clean}
The estimate
\begin{equation}
\label{bound1}
 X(C;N) \ll_{d,n} N^{2/d},
 \end{equation}
holds uniformly over all irreducible projective curves $C \subseteq \mathbb{P}^{n+1}$ of degree $d$ defined over $\mathbb{Q}$.
\end{teo}

We say a polynomial $f \in \mathbb{Z}[x_0,x_1,\ldots,x_{n+1}]$ is primitive if its coefficients have no common prime divisor. We write $X(f;N)$ for the number of points of height $\le N$ on the zero locus of $f$ and $\left\| f \right\|$ for the absolute value of the largest coefficient of $f$. Notice that we would expect $X(f;N)$ to be smaller the larger the value of $\left\| f \right\|$ is. Indeed, we have the following strengthening of Theorem \ref{clean}.

\begin{teo}
\label{main2}
The estimate
$$ X(f;N) \ll_d N^{2/d} (\log \left\| f \right\|+O_d(1))\left\| f \right\|^{-\frac{1}{d^2}}+O_d(1),$$
holds uniformly over all primitive irreducible homogeneous polynomials $f \in \mathbb{Z}[x_0,x_1,x_2]$ of degree $d$.
\end{teo}

Theorem \ref{clean} is a consequence of Theorem \ref{main2}. Indeed, this is evident when $n=1$, while the general result can be obtained from this and a projection argument contained in \cite[Theorem 4.2]{Brow} (see also \cite[\S 3]{BHBS}). On the other hand, Theorem \ref{main2} follows from applying Bezout's theorem to the $n=1$ case of the following result for hypersurfaces, which improves on a result of Salberger \cite[Theorem 1.2]{S} and will be the main goal of this article.

\begin{teo}
\label{main}
Let $f \in \mathbb{Z}[x_0,x_1,\ldots,x_{n+1}]$ be a primitive irreducible homogeneous polynomial of degree $d$ defining a hypersurface $X$. Then, there exists some homogeneous $g \in \mathbb{Z}[x_0,x_1,\ldots,x_{n+1}]$ of degree 
\begin{equation}
\label{boundf}
\ll_{d,n}  N^{\frac{n+1}{n d^{1/n}}} \left( \log \left\| f \right\|+O_{d,n}(1) \right) \left\| f \right\|^{-n^{-1}d^{-1-1/n}}+O_{d}(1),
\end{equation}
not divisible by $f$, vanishing at every rational point of $X$ of height $\le N$. 
\end{teo}

\subsection{Previous work}

The estimates given above refine or extend previous results in the literature. The well-known work of Heath-Brown \cite{HB} provides the bound
$$ X(C;N) \ll_{d,n,\varepsilon} N^{2/d+\varepsilon},$$
for every $\varepsilon > 0$ (this is obtained for plane curves in \cite{HB}, while the deduction for arbitrary $n$ is performed by Broberg in \cite{Brob}; see also \cite[Theorem 4.2]{Brow} and \cite{M}). He accomplishes this by means of a $p$-adic determinant method inspired by the corresponding work of Bombieri and Pila \cite{BP} in the affine case. The sharp estimate (\ref{clean}) was later obtained in particular cases. Browning and Heath-Brown \cite{BHB} establish this for conics, while Rault \cite{R} obtains this result for singular rational curves satisfying additional hypothesis on the resultant and discriminant.

The case of positive genus is dealt with by the work of Ellenberg and Venkatesh \cite{EV}, who establish in general the estimate
$$ X(f;N) \ll_{d,n,\varepsilon} N^{2/d+\varepsilon} (\log \left\| f \right\|+O_{d,n}(1))\left\| f \right\|^{-\frac{1}{d^2}}+O_{d}(1).$$
This motivates Theorem \ref{main2} which is a sharpening of this result. Combining this estimate with a descent on the Jacobian of the curve they obtain a power saving over the bound (\ref{bound1}) for curves of positive genus. It should be remarked that improvements for large values of $\left\|f \right\|$ are also present in other parts of the literature (e.g. \cite{HBT}).

Of particular importance to us is the remarkable work of Salberger \cite{S} establishing a conjecture of Serre. For this purpose, he introduces a 'global' determinant method and uses this to obtain the estimate
$$ X(C;N) \ll_{d,n} N^{2/d} \log N,$$
which prior to this work was the best known bound towards Theorem \ref{clean}.

Finally, we notice that Theorem \ref{main} strengthens similar results in \cite{EV,HB,S}.

\subsection{Strategy of the proof}
\label{strategy}

Although by the aforementioned results Theorem \ref{clean} was known for several families, our methods do not require any assumption on the curve. On the other hand, we make use of the global determinant method introduced by Salberger in \cite{S}, as well as an estimate of Bombieri and Vaaler \cite{BV} on the size of integer solutions to linear systems of equations. Our methods thus provide an alternative way of employing this global determinant method. We review these results in Section \ref{tools}.

The proof proceeds as follows. We are trying to find an integer polynomial vanishing at every point of $X$ of small height. This is equivalent to solving an integer system of equations, which by the result of Bombieri and Vaaler admits a solution bounded in terms of certain determinants associated with this system. If we choose a subset of these points lying in general position, then we may employ Salberger's methods to get a strong bound on these determinants. On the other hand, if every polynomial vanishing on this set is divisible by $f$ this provides us with a lower bound for the coefficients of these polynomials in terms of those of $f$. This allows us to derive an absurd upon comparison.  This argument is carried out in Section \ref{argument}.

It turns out however, that in order to implement this strategy we need to ensure that the polynomial $f$ is of an appropriate form (basically, we need the coefficient of the monomial $x_{n+1}^d$ to be essentially the largest one of $f$). In Section \ref{technical} we show how a suitable change of coordinates allows us to reduce Theorem \ref{main} to this particular case.

\medskip
\noindent \emph{Acknowledgments.} The author would like to thank Harald Helfgott and Akshay Venkatesh for providing some helpful references.

\section{Preliminaries}
\label{tools}

We use the usual asymptotic notation $X=O(Y)$ to mean that $|X| \le C |Y|$ for some constant $C$, and we may also write $X \ll Y$ instead of $X=O(Y)$. We will usually deal with an integer parameter $N$ going to infinity, in which case the notation $o(X)$ represents a quantity $Y$ such that $|Y|/|X| \rightarrow 0$ as $N \rightarrow \infty$. The implicit constants in the above notations will only be allowed to depend on the parameters $d,n$ in Theorem \ref{main}, and this dependence will be indicated by a subscript. Also, we use the notation $X \sim Y$ for $X \ll Y \ll X$.

In this section we recall some results that will be needed in the rest of the article.

\subsection{Linear equations over the integers}

We begin with the following refinement of the Thue-Siegel lemma due to Bombieri and Vaaler.

\begin{teo}{\cite[Theorem 1]{BV}}
\label{BV}
Let $\sum_{k=1}^r a_{mk} x_k=0$, $m=1,\ldots,s$, be a system of $s$ linearly independent equations in $r>s$ unknowns, with integer coefficients $a_{mk}$. Then, there exists a nontrivial integer solution $(x_1,\ldots,x_r)$ satisfying the bound
$$ \max_{1 \le i \le r} |x_{i}| \le  \left( D^{-1} \sqrt{\left|\det \left( A A^T \right) \right|} \right)^{\frac{1}{r-s}}.$$
Here $A=(a_{mk})$ is the $s \times r$ matrix of coefficients, $A^T$ its transpose, and $D$ is the greatest common divisor of the determinants of the $s \times s$ minors of $A$.
\end{teo}

The article \cite{BV} contains a number of generalizations of this result concerning basis of solutions and arbitrary number fields. For the purposes of the present note it may be relevant to mention that Theorem \ref{BV} as stated allows a simpler proof than these general results.

\subsection{A global determinant method}

We now turn to the second result we shall need. Let $X$ be a hypersurface defined by a primitive absolutely irreducible homogeneous polynomial $f \in \mathbb{Z}[x_0,x_1,\ldots,x_{n+1}]$. A crucial ingredient of our proof will be an estimate on determinants due to Salberger (results of this kind date back to the work of Heath-Brown \cite{HB}).

\begin{teo}{\cite[Lemmas 1.4 and 1.5]{S}}
\label{Sal1}
Let $p$ be a prime for which the reduction of $X$ is absolutely irreducible. Let $(\xi_1,\ldots,\xi_s)$ be a tuple of rational points in $X$, let $F_1,\ldots,F_s \in \mathbb{Z}[x_0,x_1,\ldots,x_{n+1}]$ be forms with integer coefficients and write $\Delta$ for the determinant of $\left( F_i(\xi_j) \right)_{ij}$. Then, there exists some 
$$e \ge n!^{1/n} \frac{n}{n+1} \frac{s^{1+1/n}}{p+O_{d,n}\left(p^{1/2} \right)}+O_{d,n}(s),$$
such that $p^e|\Delta$.
\end{teo}

We derive from this the following corollary.

\begin{teo}[{cf. \cite[proof of Theorem 1.2]{S}}]
\label{Sal2}
Let $(\xi_1,\ldots,\xi_s)$ be a tuple of rational points in $X$, let $F_{li} \in \mathbb{Z}[x_0,x_1,\ldots,x_{n+1}], 1 \le l \le L, 1 \le i \le s,$ be forms with integer coefficients and write $\Delta_l, 1 \le l \le L,$ for the determinant of $\left( F_{li}(\xi_j) \right)_{ij}$. Let $\Delta$ be the greatest common divisor of the $\Delta_l$. Then, we have the bound
$$ \log |\Delta| \ge \frac{n!^{1/n}}{n+1} s^{1+1/n} \left( \log s +O_{d,n}(1)-n \max \left\{ \log \log \left\| f \right\|, 0 \right\} \right).$$
\end{teo}

\begin{proof}
Write $\mathcal{P}_X$ for the set of primes for which the reduction of $X$ is not absolutely irreducible. By a result of Noether (see \cite[\S V.2, Theorem 2 A]{Sch} or \cite[Lemma 1.8]{S}), given $d$ and $n$ there exist universal forms $\Phi_1,\ldots,\Phi_l$ with integer coefficients, such that a form $F \in K[x_0,x_1,\ldots,x_{n+1}]$ of degree $d$ is absolutely irreducible over a field $K$ if and only if not all these universal forms vanish when evaluated at the coefficients of $F$ (if $K$ has positive characteristic, the corresponding reduction is to be applied to the universal forms). Since $f$ is absolutely irreducible, the value of some $\Phi_i$ at its coefficients must be nonzero and divisible by every $p \in \mathcal{P}_X$. It follows that
\begin{equation}
\label{venka}
 \prod_{p \in \mathcal{P}_X} p \ll_{d,n} \left\| f \right\|^{O_{d,n}(1)}.
 \end{equation}
Applying Theorem \ref{Sal1} for each prime $p \le s^{1/n}, p \notin \mathcal{P}_X$, we see that $\log |\Delta|$ is bounded from below by
$$  \frac{n!^{1/n}.n}{n+1} s^{1+1/n} \sum_{p \le s^{1/n}, p \notin \mathcal{P}_X} \frac{\log p}{p+O_{d,n}\left(p^{1/2} \right)} - O_{d,n}(s) \sum_{p \le s^{1/n}} \log p.$$
Using (\ref{venka}), the classical facts
\begin{equation}
\label{classic1}
\sum_{p \le x} \frac{\log p}{p} = \log x+O(1), \, \, \sum_{p \le x} \log p \sim x,
\end{equation}
and the estimate 
\begin{equation}
\label{32}
\frac{\log p}{p+O_{d,n}(p^{1/2})}-\frac{\log p}{p} \ll_{d,n} \frac{\log p}{p^{3/2}} ,
\end{equation}
we see that
$$ \sum_{p \in \mathcal{P}_X} \frac{\log p}{p+O_{d,n}(p^{1/2})} \le \max \left\{ \log \log \left\| f \right\|, 0 \right\} + O_{d,n}(1).$$
Combining this with our previous bound and using the estimates (\ref{classic1}), (\ref{32}) again, we conclude that $\log |\Delta|$ is at least
$$ \frac{n!^{1/n}}{n+1} s^{1+1/n} \left(  \log s +O_{d,n}(1)-n \max \left\{ \log \log \left\| f \right\|,0 \right\} \right),$$
as claimed.
\end{proof}

\section{Change of coordinates}
\label{technical}

In order to carry out the argument outlined in Section \ref{strategy} we will need to assume the coefficient of the monomial $x_{n+1}^d$ is essentially the largest one of $f$. The reduction of Theorem \ref{main} to this particular case is performed in the present section. With this purpose in mind we fix an arbitrary choice of $d$ and $n$. We will write $\mathcal{U}$ for the class of all primitive irreducible homogeneous polynomials $f \in \mathbb{Z}[x_0,x_1,\ldots,x_{n+1}]$ of degree $d$, and we write $\mathcal{V}$ for the set of polynomials $f \in \mathcal{U}$ satisfying
\begin{equation}
\label{V}
 c_f \gg_{d,n} \left\| f \right\|,
 \end{equation}
where the implicit constant is independent of $f$ and shall be specified later. Here $c_f$ stands for the coefficient of the monomial $x_{n+1}^d$ in $f$, while we recall that the norm $\left\| f \right\|$ gives the size of the largest coefficient of a monomial appearing in $f$.

\begin{prop}
\label{cc}
Suppose Theorem \ref{main} holds uniformly over all $f \in \mathcal{V}$. Then it also holds uniformly over all $f \in \mathcal{U}$, albeit with a possibly worst implicit constant.
\end{prop}

The rest of this section is devoted to the proof of this proposition, so we will assume from now on that Theorem \ref{main} holds over every $f \in \mathcal{V}$. We begin with the following easy lemma that allows us to apply a bounded change of coordinates to our functions.

\begin{lema}[Change of coordinates]
\label{change}
Let $A \in \text{SL}_{n+2}(\mathbb{Z})$ be a matrix with all its coefficients of size $O_{d,n}(1)$. Let $f \in \mathcal{U}$ be arbitrary. Then, if $g \in \mathbb{Z}[x_0,x_1,\ldots,x_{n+1}]$ is not divisible by $f \circ A$ and vanishes at every rational zero of $f \circ A$ of height $O_{d,n}(N)$, then $g \circ A^{-1}$ is not divisible by $f$ and vanishes at every rational zero of $f$ of height $\le N$.
\end{lema}

\begin{proof}
It is clear that $f|g \circ A^{-1}$ would imply $f \circ A | g$. It will therefore suffice to show that $A^{-1}$ gives an injection from the set of primitive solutions of $f$ of height at most $N$ to the set of primitive solutions of $f \circ A$ of height at most $O_{d,n}(N)$, and this follows immediately from the boundedness of the coefficients of $A^{-1}$.
\end{proof} 

We will need the following simple observation to ensure that we can find a change of coordinates of an appropriate form.

\begin{lema}
\label{box}
Suppose $f \in \mathbb{C}[x_0,\ldots,x_{n+1}]$ is an homogeneous polynomial of degree $d$ vanishing at $(t_0,t_1,\ldots,t_n,1)$ for every choice of integers $|t_i|=O_{d,n}(1)$. Then $f$ must be zero.
\end{lema}

\begin{proof}
Since $x_{n+1}^d$ fails to vanish at these points, it will suffice to show that there is an $(n+1)$-dimensional box of size $O_{d,n}(1)$ such that there is no nontrivial polynomial $g \in \mathbb{C}[x_0,\ldots,x_n]$ (not necessarily homogeneous) of degree at most $d$ vanishing at every integer point of this box. Let $\kappa$ be the application which associates to each $y \in \mathbb{Z}^{n+1}$ the tuple $\kappa(y)$ consisting of the values it takes over the monomials in $n+1$ variables of degree at most $d$. Since there are $O_{d,n}(1)$ such monomials, we see that there is a maximal linearly independent set $\kappa(y_1),\kappa(y_2),\ldots,\kappa(y_l)$ with $l=O_{d,n}(1)$. In particular, this implies that every polynomial vanishing at $y_i$ for every $1 \le i \le l$ must vanish at $\mathbb{Z}^{n+1}$ and therefore be zero. Choosing a box large enough to contain these points we get the desired result.
\end{proof}

Let $f \in \mathcal{U}$ be given. Our next step is to find some $A$ as in Lemma \ref{change} such that $f \circ A$ has a large coefficient at $x_{n+1}^d$. Since this coefficient will equal $f(a_0,a_1,\ldots,a_{n+1})$ with $a_i$ the coefficients of the last column of $A$, we want to show that there is some choice of integers $a_0,a_1,\ldots,a_{n+1} = O_{d,n}(1)$ such that
$$ f(a_0,a_1,\ldots,a_{n+1}) \gg_{d,n} \left\| f \right\|.$$
We do this in the next lemma.

\begin{lema}
\label{compact}
Let $T$ be the set of tuples of integers $(a_0,a_1,\ldots,a_n,1)$ with $a_i=O_{d,n}(1)$ for every $i$. Then, for every homogeneous $g \in \mathbb{C}[x_0,\ldots,x_{n+1}]$ of degree $d$, there exists some $a \in T$ such that
$$ |g(a)| \gg_{d,n} \left\| g \right\|.$$
\end{lema}

\begin{proof}
By a simple rescaling argument, it will actually suffice to show that
$$ \inf_{\left\| g \right\|=1} \max_{a \in T} |g(a)| \gg_{d,n} 1,$$
where the infimum is taken over all homogeneous $g$ of degree $d$ with $\left\| g \right\|=1$. Let $\mathbb{C}^m$ with $m=\binom{d+n+1}{n+1}$ be the space of coefficients of the homogeneous polynomials of degree $d$, so that $\left\| g \right\|=1$ corresponds to the standard sphere in the $L^{\infty}$-norm. It is clear that the application that sends each $m$-tuple of coefficients to $\max_{a \in T} |g(a)|$ is continuous, with $g$ the polynomial associated with the tuple of coefficients. By Lemma \ref{box} we know that $\max_{a \in T} |g(a)|>0$ for every nonzero $g$ of degree $d$. By compactness of the sphere, it thus also follows that
$$ \inf_{\left\| g \right\|=1} \max_{a \in T} |g(a)| \ge C,$$
for some constant $C$ depending only on $d$ and $n$, and this gives us the desired result.
\end{proof}

Given $f \in \mathcal{U}$, choose a tuple of coefficients ${\bf a}=\left\{ a_0,a_1,\ldots, a_n,1 \right\}$ of size $O_{d,n}(1)$ as in Lemma \ref{compact}. It follows that the matrix $A=I+B$ satisfies the conditions of Lemma \ref{change}, where $I$ is the identity matrix and $B$ is the matrix with $\left\{ a_0,a_1,\ldots, a_n,0 \right\}$ in the last column and $0$ everywhere else. But from the choice of ${\bf a}$ we also see that $f \circ A$ is such that the monomial $x_{n+1}^d$ has its coefficient of size $\gg_{d,n} \left\| f \right\|$. We deduce from this and the boundedness of the coefficients of $A$ that
\begin{equation}
\label{order}
 \left\| f \circ A \right\| \sim_{d,n} \left\| f \right\|,
 \end{equation}
from where we see that $f \circ A$ satisfies (\ref{V}) upon an appropriate choice of the implicit constant. 

It only remains to show that $f \circ A$ is primitive. Let $m$ be the greatest common divisor of the coefficients of $f \circ A$. Then $\frac{1}{m} \left( f \circ A \right)$ has integer coefficients, which clearly implies that $\frac{1}{m} \left( f \circ A \right) \circ A^{-1}=\frac{1}{m} f$ must also have integer coefficients. Since $f$ was already primitive, it must necessarily be $m=1$ as claimed.

We can now complete the proof of Proposition \ref{cc}. Given $f \in \mathcal{U}$, we know that there is a choice of $A$ as above such that $f \circ A \in \mathcal{V}$. Let $g$ be a polynomial as in Theorem \ref{main} vanishing on the rational zeros of $f \circ A$ of height $O_{d,n}(N)$, so that in particular the degree of $g$ is at most
\begin{align*}
&\ll_{d,n} N^{\frac{n+1}{n d^{1/n}}} \frac{\left( \log \left\|  f \circ A \right\|+O_{d,n}(1) \right)}{\left\| f \circ A \right\|^{n^{-1}d^{-1-1/n}} }+O_d(1) \\
&\ll_{d,n}  N^{\frac{n+1}{n d^{1/n}}} \frac{\left( \log \left\| f \right\|+O_{d,n}(1) \right)}{\left\| f \right\|^{n^{-1}d^{-1-1/n}}}+O_d(1),
\end{align*}
where we have used (\ref{order}). It follows from Lemma \ref{change} that $g \circ A^{-1} \in \mathbb{Z}[x_0,\ldots,x_{n+1}]$ is an homogeneous polynomial, not divisible by $f$, vanishing at the rational zeros of $f$ of height at most $N$. But since the degree of $g \circ A^{-1}$ coincides with that of $g$, we conclude that Theorem \ref{main} holds over every $f \in \mathcal{U}$ as desired.

\section{Proof of Theorem 1.3}
\label{argument}

We now proceed to the proof of Theorem \ref{main}. By the results of the previous section, we know that it will suffice to establish this over every $f \in \mathcal{V}$ and therefore we assume that $f$ lies on this space from now on. 

Assume first that $f$ is reducible over $\overline{\mathbb{Q}}$ and let $h$ be a nontrivial irreducible factor. We may write $h=\sum_{j=1}^k \lambda_j h_j$, where the $h_j$ are integer forms of the same degree, while the coefficients $\lambda_j$ are linearly independent over the rationals. Clearly, $h$ does not divide $h_1$ since, by the irreducibility of $f$ over $\mathbb{Q}$, it cannot be proportional to a form with integer coefficients. On the other hand, $h_1$ is an homogeneous form with integer coefficients that vanishes at every rational zero of $h$, by the linear independence of the $\lambda_j$. Repeating this argument for every nontrivial irreducible factor of $f$ and taking products, we obtain a form of degree $d$ and with integer coefficients, which is not divisible by $f$ and vanishes at every rational zero of $f$, yielding Theorem \ref{main} in this case. We may thus assume from now on that $f$ is absolutely irreducible.

We write $S$ for the set of rational solutions of $f$ of height at most $N$. We will let $M$ be an integer of size
\begin{equation}
\label{M}
 M \sim_{d,n}  N^{\frac{n+1}{n d^{1/n}}} \frac{\left( \log \left\| f \right\|+O_{d,n}(1) \right)}{\left\| f \right\|^{n^{-1}d^{-1-1/n}}},
 \end{equation}
 with the precise value of the implicit constant to be specified below. This integer will correspond to the degree of the polynomial $g$ we shall construct. Write ${\bf \xi}=\left\{ \xi_1,\ldots,\xi_s \right\}$ for a maximal set of elements of $S$ which is algebraically independent over polynomials of degree $M$, in the sense that if we let $A$ be the matrix for which the $i$th row is given by the values of $\xi_i$ at the different homogeneous monomials of degree $M$, then this matrix has rank $s$.

Given an integer $D$, we write $\mathcal{B}[D]$ for the set of homogeneous monomials of degree $D$, so that in particular $|\mathcal{B}[D]| = \binom{D+n+1}{n+1}$. Since the elements of $f . \mathcal{B}[M-d]$ provide linearly independent solutions vanishing on $S$, it follows that the rank $s$ of $A$ satisfies
$$ s \le |\mathcal{B}[M]|-|\mathcal{B}[M-d]|.$$
On the other hand, since $A$ has integer coefficients, if this were a strict inequality we would be able to find a polynomial in the $\mathbb{Z}$-span of $\mathcal{B}[M]$ that vanishes on $\xi$ (and therefore on $S$, by the maximality of $\xi$) but that is not divisible by $f$, concluding the proof of our result.  We may thus assume from now on that 
\begin{equation}
\label{s}
 s=|\mathcal{B}[M]|-|\mathcal{B}[M-d]|,
 \end{equation}
and we notice in particular that this implies
\begin{equation}
\label{as}
 s=\frac{dM^n}{n!}+O_{d,n} \left( M^{n-1} \right).
 \end{equation}
We will compare this with the results of Section \ref{tools} to derive a contradiction (for a sufficiently large implicit constant in (\ref{M})) and therefore obtain the desired result. 

Recall that we write $A$ for the matrix for which the $i$th row is given by the values of $\xi_i$ at the different homogeneous monomials of degree $M$. Write $D$ for the greatest common divisor between the determinants of all $s \times s$ minors of $A$. We are trying to find integer solutions to the system of linear equations given by $A$. By Theorem \ref{BV} we know that there exists some nonzero homogeneous polynomial of degree $M$ vanishing on $\xi$ with integer coefficients $\left( x_{1},\ldots,x_{r} \right)$ satisfying the inequality
\begin{equation}
\label{frame}
 \left( \max_{1 \le i \le r} |x_{i}| \right)^{r-s} \le D^{-1} \sqrt{\left|\det \left( A A^T \right) \right|} ,
 \end{equation}
with $r=|\mathcal{B}[M]|$, so in particular  
$$r-s=|\mathcal{B}[M-d]| = \frac{M^{n+1}}{(n+1)!}+O_{d,n}(M^n).$$
Since each coefficient of $A$ is bounded by $N^M$, we know that 
$$ \left|\det \left( A A^T \right) \right| \le s!(|\mathcal{B}[M]|N^{2M})^s.$$
Hence (by Stirling's formula and the bound $|\mathcal{B}[M]| \ll_{d,n} M^{n+1}$),
\begin{align}
 \log \sqrt{\left|\det \left( A A^T \right) \right|} &\le \frac{s+1}{2} \log s + \frac{(n+1)}{2}s \log M+Ms \log N+O_{d,n}(1) \notag \\
 \label{e3}
 &\le (n+2)s \log s+Ms \log N+O_{d,n}(1). 
 \end{align}
On the other hand, we may apply Theorem \ref{Sal2} over the set of $s \times s$ minors of $A$ to conclude that
\begin{equation}
\label{e2}
\log |D| \ge \frac{n!^{1/n}}{n+1} s^{1+1/n} \left( \log s +O_{d,n}(1)-n \max \left\{ \log \log \left\| f \right\|, 0 \right\} \right).
 \end{equation}
Finally, we recall that by the assumption (\ref{s}) we know that every homogeneous polynomial $P$ of degree $M$ with integer coefficients vanishing on ${\bf \xi}$ must be of the form $P=fh$ for some homogeneous polynomial $h$ of degree $M-d$, which must have integer coefficients by Gauss's lemma. Let $W$ be the greatest monomial (in right to left lexicographical order) appearing in $h$ with a nonzero coefficient and let $w$ be this coefficient. It is then clear that the monomial $x_{n+1}^dW$ appears in $P$ with $c_f w$ as coefficient, where $c_f$ is the coefficient of $x_{n+1}^d$ in $f$. Since $f \in \mathcal{V}$ and $w \in \mathbb{Z} \setminus \left\{ 0 \right\}$, we see that
$$ \left\| P \right\| \ge |c_f w| \gg_{d,n} \left\| f \right\|.$$
But this holds for every homogeneous polynomial $P$ of degree $M$ with integer coefficients vanishing on $\xi$, from where it follows that the left hand side of (\ref{frame}) is at least 
$$ \left( c_{d,n} \left\| f \right\| \right)^{\frac{M^{n+1}}{(n+1)!}+O_{d,n}(M^n)},$$
 for some constant $c_{d,n}>0$ depending only on $d$ and $n$.

Comparing with (\ref{e3}) and (\ref{e2}), we must have
\begin{align*}
\frac{n!^{1/n}}{n+1} s^{1+1/n} &\left( \log s +O_{d,n}(1)-n \max \left\{ \log \log \left\| f \right\|, 0 \right\} \right) \\
&\le Ms \log N -\left(\frac{M^{n+1}}{(n+1)!}+O_{d,n}(M^n) \right) \log \left( c_{d,n} \left\| f \right\| \right).
\end{align*}
It is clear that if $\left\| f \right\|$ is a large power of $N$ then $\frac{M^{n+1}}{(n+1)!} \log \left\| f \right\|$ will dominate and we derive a contradiction. Otherwise we see from (\ref{M}) and (\ref{as}) that $O_{d,n}\left(M^n \log \left( c_{d,n} \left\| f \right\| \right) \right) = o_{d,n}(s^{1+1/n})$ and we get the inequality,
\begin{align*}
 \frac{n!^{1/n}}{n+1} s^{1+1/n} &\left( \log s +O_{d,n}(1)-n \max \left\{ \log \log \left\| f \right\|, 0 \right\} \right) \\
 &\le Ms \log N -\frac{M^{n+1}}{(n+1)!} \log \left( c_{d,n} \left\| f \right\| \right).
 \end{align*}
Dividing everything by $Ms$ and using (\ref{as}) again, we get 
\begin{equation}
\label{absurd}
 \frac{nd^{1/n}}{n+1} \left( \log M +O_{d,n}(1)- \max \left\{ \log \log \left\| f \right\|, 0 \right\} \right) \le \log N-\frac{\log \left\| f \right\|}{d(n+1)},
 \end{equation}
where we have used the estimates  
\begin{align*}
M^{-1} \log \left( c_{d,n} \left\| f \right\| \right)&=o_{d,n}(1) \\
 \frac{s^{1/n}}{M} = \frac{d^{1/n}}{n!^{1/n}}&+O_{d,n} \left( M^{-1/n} \right).
 \end{align*}
Since (\ref{absurd}) can clearly be contradicted upon choosing the implicit constant in (\ref{M}) sufficiently large, this concludes the proof of Theorem \ref{main}.

\end{document}